\documentclass{amsart}
\usepackage{amsthm, amsfonts, amssymb, latexsym, color, tikz}
\usepackage{amsmath,amssymb,amsbsy,amsfonts,latexsym,amsopn,amstext, amsxtra,euscript,amscd,bm}
\usepackage[ruled]{algorithm2e}
\usepackage{enumerate,enumitem}

\usepackage{hyperref}
\usepackage{lipsum}
\usepackage{fixme}
\fxsetup{
    status=draft,
    author=,
    layout=margin,
    theme=color
}

\newcommand{\E}{\mathsf{E}}
\newcommand{\Tr}{\mathsf{Tr}}

\newcommand{\F}{\mathbb{F}}

\usepackage{amsmath}
\usepackage[left=1in,right=1in,top=1.5in,bottom=1.5in]{geometry} 
\newtheorem{proposition}{Proposition}
\newtheorem*{proposition*}{Proposition}

\newtheorem{lemma}{Lemma}
\newtheorem*{corollary*}{Corollary 5}
\newtheorem{definition}{Definition}
\newtheorem{theorem}{Theorem}

\newtheorem{claim}{Claim}

\title{Low-energy decomposition results over finite fields}

\author[A.\ Mohammadi]{Ali Mohammadi}

\address{A.M.: School of Mathematics, Institute for Research in Fundamental Sciences (IPM),
Tehran, Iran.}
\email{a.mohammadi@ipm.ir}

 \author[S.\ Stevens]{Sophie Stevens}
\address{S.S.: Johannn Radon Institute for Computational and Applied Mathematics (RICAM), Linz, Austria}
\email{sophie.stevens@oeaw.ac.at}

\begin{document}
\maketitle
\begin{abstract}
   We prove various low-energy decomposition results, showing that we can decompose a finite set $A\subset \mathbb{F}_p$ satisfying $|A|<p^{5/8}$, into $A = S\sqcup T$ so that, for a non-degenerate quadratic $f\in \F_p[x,y]$, we have
    \[ |\{(s_1,s_2,s_3,s_4)\in S^4 : s_1 + s_2 = s_3 + s_4\}| \ll  |A|^{3 - \frac15 + \epsilon}
    \]
    and
    \[
    |\{(t_1,t_2,t_3,t_4)\in T^4 : f(t_1, t_2) = f(t_3, t_4)\}|\ll |A|^{3 - \frac15 + \epsilon}\,.
    \] 

    Variations include extending this result to large $A$  and a low-energy decomposition involving additive energy of images of rational functions.  This gives an improvement to a result of Roche-Newton, Shparlinski and Winterhof \cite{RNShWin} 
    as well as a generalisation of a result of Rudnev, Shkredov and Stevens \cite{RudShSt}. 
    
    We consider applications to conditional expanders,  exponential sum estimates and the finite field Littlewood problem. In particular, we improve results of Mirzaei \cite{Mirzaei}, Swaenepoel and Winterhof \cite{SwaWin} and Garcia \cite{Garc}.
\end{abstract}
\section{Introduction}

In this paper, we show results of the following flavour: for a suitable non-degenerate polynomial $f\in \mathbb{F}_p[x,y]$, $A\subseteq \mathbb{F}_p$ admits a decomposition into disjoint sets $A = S\sqcup T$ so that 
\[
\E(S):=|\{(s_1,s_2,s_3,s_4)\in S^4 : s_1 + s_2 = s_3 + s_4\}| \ll |A|^{3 - \frac15 + \epsilon}
\]
and
\[
\E_f(T):=|\{(t_1,t_2,t_3,t_4)\in T^4 : f(t_1, t_2) = f(t_3, t_4)\}|\ll |A|^{3 - \frac15 + \epsilon}\,.
\] 

The actual formulation of our results is somewhat technical, and so we defer their formal presentation until Section~\ref{sec:main results}, after we have developed the necessary terminology and context of these results.

\subsection{Background}

 The sum-product problem over finite fields is a quantitative interpretation of the observation that a set $A\subseteq \mathbb{F}_q$ (where $q$ is some prime power)  cannot be both additive and multiplicative, unless the intersection of $A$ with a multiplicative coset of a subfield of $\mathbb{F}_q$ is large. More precisely, the Erd\H{o}s-Szemer\'edi conjecture (over $\mathbb{F}_q$) asks if, for every $0<\epsilon <1$, the inequality
 \begin{equation}
 \label{eqn:SPConjecture}
     \max\{|A+A|, |AA|\}\gg_{\epsilon} |A|^{1+\epsilon}
 \end{equation}
holds under some natural conditions on sets $A\subset\F_q$. 

Here, we denote by $A+A$ the sum set of $A$: given $A, B\subseteq \F_q$, we define 
\[A+B = \{a+b:(a,b)\in A\times B\}\,.\]
We similarly define $A-B$, $AB$ and $A/B$, which we refer to as the difference, product and ratio sets of $A$ and $B$ respectively; we do not consider division by zero in the definition of $A/B$.

Note, for instance, that no $\epsilon>0$ exists in the case that $A=cG$ for some subfield $G$ of $\F_q$ and element $c\in \F_q$. 
The most studied instances of this problem involve either sets that are large in terms of the order of the field (e.g. $|A| > q^{1/2}$; see Garaev~\cite{Gar} for the state-of-the-art in this direction) or sets that are small in terms of the characteristic of the field (e.g. $|A|<p^{1/2}$; see the authors' companion paper \cite{MohSte} for the best results towards this problem). 
A work of Roche-Newton and Li \cite{RNLi} considers a less restrictive constraint, based on the size of the intersection of the set in question and multiplicative cosets of proper subfields of $\mathbb{F}_q$.

We describe $A$ as e.g. `additive' if its sum set is small, the canonical example being an arithmetic progression, where $|A+A|= 2|A|-1$. However this classification is fragile, for example, if $A$ is the union of an arithmetic progression of size $N$, and a random set of $N$ elements, then $|A+A|\sim |A|^2$, but it is clear that $A$ does possess additive structure (in the sense that $A$ contains a large arithmetic progression). 
A more robust characterisation of structure can be given via the energy of a set: for sets $A$ and $B$ we define the representation functions 
\begin{equation}
\label{eqn:1stsndmEn}
    r_{A\circ B}(\lambda) = |\{(a, b)\in A\times B: a\circ b = \lambda\}| \quad \text{for}\quad \circ\in \{+, -, \times, /\}\,.
\end{equation}
The additive and multiplicative energies of $A$ and $B$ are moments of the representation functions: for $k\geq 1$ we define
\[
\E_k(A,B):= \sum_{x\in A-B}r_{A-B}^k(x) \text{ and } \E_{k}^{\times}(A,B):= \sum_{x\in A/B}r_{A/B}^k(x)\,.
\]
We write $\E_k(A,A) = \E_k(A)$ and $\E(A,B) = \E_2(A,B)$; we do the same for multiplicative energy. . 
When $k\in \mathbb{N}$, the energy has a combinatorial interpretation as the number of solutions to (in the case of additive energy) the equation $a_1 - b_1 = \dots = a_k - b_k$ where $a_i\in A$ and $b_i\in B$. 
The case $k = 2$ is special because we can replace ``$+$" with ``$-$" in the above. 
Hence, through a simple application of the Cauchy-Schwarz inequality, we have 
 \begin{equation}
 \label{eqn:AddEnCS}
     \E(A, B) |A\pm B| \geq |A|^2|B|^2.
 \end{equation}
 These inequalities also hold for their multiplicative counterparts.
 
\subsection{Energy formulation of the sum product problem}
 Balog and Wooley~\cite{BW} raised the question of whether, in line with the sum-product problem, a similar duality statement exists for the energies of a given set, i.e. additive and multiplicative structure cannot coexist in a set, as measured by the energy. 
 In \cite[Theorem~1.3]{BW}, the authors prove that for $A\subset \F_p$, there exist disjoint subsets $B, C\subseteq A$ such that $A=B\cup C$ and
 \begin{equation}
 \label{eqn:LEDMEx}
     \max\{\E(B), \E^\times(C)\} \lesssim |A|^{3-\frac4{101} } + |A|^{3+\frac1{15}}p^{-\frac1{15}}\,.
 \end{equation}
This decomposition formulation is necessary: for example, if $A$ is the union of an arithmetic progression and a geometric progression of the same size, then both of its energies are essentially maximally large. Since $\max\{|B|, |C|\} \geq |A|/2$, the Cauchy-Schwarz estimate  \eqref{eqn:AddEnCS} (or its multiplicative counterpart) converts a statement of the form \eqref{eqn:LEDMEx} into a sum-product inequality.
A significant quantitative improvement to \eqref{eqn:LEDMEx} was obtained by Rudnev, Shkredov and Stevens \cite[Theorem~2.8]{RudShSt}, who proved that when $|A|\leq p^{5/8}$, one can take $|A|^{3-1/5}$ as the upper bound.

\subsection{Energy formulation of sums and function images problem}

 A variation of the sum-product problem, considered for example by Bukh and Tsimerman \cite{BukhTsim}, Solymosi~\cite{Sol08} and  Cilleruelo et al. \cite{CGOS}, is to establish a non-trivial lower bound for the quantity $\max\{|A+A|, |f(A)+f(A)|\}$ under some natural conditions on $A$ and some function $f$. The energy-variant of this problem was considered by Roche-Newton, Shparlinski and Winterhof \cite{RNShWin}, which we now describe. 
 
 Given a rational function $f\not\equiv 0$ in $\F_q(x)$, we write $f = g/h$, for coprime polynomials $g$ and $h\not\equiv 0$ in $\F_q[x]$ and define the degree of $f$ to be $d = \max\{\text{deg}(g), \text{deg}(h)\}$. 
 We say that a rational function $f$ in $\F_q(x)$ of degree $d$ is \emph{non-degenerate} if
 \begin{equation}
     \label{eqn:RNSHWINDC}
 f(x)\not\in \{(a(g(x)^p -g(x)) + b x + c: g(x)\in \F_q (x); a,b,c\in \F_q\}\,.
  \end{equation}

 Roche-Newton et al. \cite[Theorem~1.1]{RNShWin} show that for any set $A\subseteq\F_q$ and degree $d$ rational function  $f$ in $\F_q(x)$, there exists a decomposition of $A$ into disjoint subsets $S, T\subseteq A$ such that
\[
\max\{\E(S), \E (f(T))\}\ll_d \frac{|A|^3}{M_A},
\]
where 
\begin{equation}
\label{eqn:RNSHWIMExp}
M_A = \min\Bigg\{\frac{q^{1/2}}{|A|^{1/2}(\log|A|)^{11/4}}, \frac{|A|^{4/5}}{q^{2/5}(\log|A|)^{31/10}}\Bigg\}.
\end{equation}

This estimate is non-trivial when $|A|\gtrsim q^{1/2}$ ; no similar results are known for smaller sets. Macourt \cite{Mac} has studied  variations of \cite[Theorem~1.1]{RNShWin}. The non-degeneracy condition on $f$ ensures that it is not a linearised permutation polynomial (see \cite[Section~1.3]{RNShWin}).

\subsection{Function energy formulation}

Both the sum-product problem and its function-image variant discussed above are instances of the more general problem of the growth of $|A+A|+|f(A, A)|$ for some bivariate function $f$. See for example \cite{BukhTsim}, \cite{KoNassPhamVal}, \cite{Mirzaei} and \cite{Vu} for results in this direction. Note that, if $f = g(ax+by)$ for some polynomial $g$ and $A$ is an arithmetic progression, then $|A+A|\approx |f(A, A)| \approx |A|$. Hence, we require the following notion of non-degeneracy. 
\begin{definition}
\label{def:DegPoly}
We say a polynomial $f \in \F_q\left[x, y\right]$ is non-degenerate if it depends on each variable and is not of the form $g(ax+by)$, for some univariate polynomial $g$ with coefficients in $\F_q$.
\end{definition}

In this paper, we consider the energy formulation for this more general problem. 

Given a bivariate polynomial $f\in \F_q\left[x, y\right]$, define 
\begin{equation}
\label{eqn:bvEndef}
    \E_{f}(A, B):=|\{(a_1, a_2, b_1, b_2)\in A^2\times B^2: f(a_1, b_1) = f(a_2, b_2)\}|.
\end{equation}
For $\lambda\in \F_q$, let 
\begin{equation*}
r_{f(A, B)}(\lambda) = |\{(a, b)\in A\times B: f(a, b) = \lambda\}|.
\end{equation*}
We record the following identities and consequence of the Cauchy-Schwarz inequality:
\begin{equation*}
 \sum_{\lambda\in \F_q}r_{f(A, B)}(\lambda) = |A||B|\text{ ,}\quad     \sum_{\lambda\in \F_q}r_{f(A, B)}(\lambda)^2 = \E_{f}(A, B)\text{ , }\quad
    \E_{f}(A, B)|f(A, B)| \geq |A|^2|B|^2.
\end{equation*}
Roche-Newton et al. \cite[Question~5.1]{RNShWin} ask if one can prove the existence of a decomposition $A= S\cup T$, such that
\begin{equation}
\label{eqn:RSWQ}
    \max\{\E_{f}(S), \E_{g}(T)\}\ll |A|^{3-\delta},
\end{equation}
for some $\delta >0$, where $A\subset \F_q$ and $f, g \in \F_q\left[x, y\right]$ satisfy some natural non-degeneracy conditions. In this paper, incorporating some ideas from \cite{KoNassPhamVal}, we prove results of this nature.

\subsection{Applications of low-energy decomposition results}
Progress on low-energy decomposition results has led to direct improvements to the Erd\H{o}s-Szemer\'edi sum-product problem \eqref{eqn:SPConjecture}. Due to the increased level of attention that this problem attracts, we have recorded the improvements to the finite field sum-product problem due to the techniques here in a companion paper \cite{MohSte}.

The (bivariate) function variation of low-energy decomposition results similarly leads to progress on the
 bivariate function variation of the sum-product problem. Current progress on this problem is summarised by a result of Mirzaei~\cite{Mirzaei} who shows that $A\subseteq \mathbb{F}_p$ with $|A|<p^{1/2}$ satisfies
\[
\max\{|A\pm A|,|f(A,A)|\} \gtrsim |A|^{\frac65 + \frac4{305}}\,.
\]
Using low-energy decomposition statements, we improve the exponent in the case of difference sets. We expect these techniques to lead to further progress on related problems in the expanders literature, for example, the unconditional growth of the images of sets under specific polynomials. 

As observed by Balog and Wooley~\cite{BW} as well as in the later works \cite{RNShWin, SwaWin}, low-energy decomposition results provide useful tools in showing cancellation amongst various types of character sums. 

Let $q=p^n$ and define 
\[e_p(x) := \exp(2\pi i x/p)\quad \text{and} \quad\psi(x) = e_p(\Tr(x))\,,\] 
where $\Tr(x) = x+x^{p} + \cdots + x^{p^{n-1}}$ is the trace of $x\in \F_{q}$ over $\F_p$. 

Given sets $S, T\subseteq \F_q$, Vinogradov (see \cite[p. 92]{Vino}) showed that
\begin{equation}
\label{eqn:VinoB}
    \bigg|\sum_{s\in S}\sum_{t\in T}\psi(st)\bigg|\leq \sqrt{|S||T|q}.
\end{equation}
This bound is non-trivial if $|S||T|>q^{1/2}$. 
There are many results that improve \eqref{eqn:VinoB}, either in terms of strength of the bound or its effective range (see e.g. \cite{BouGar14, BouGlibKon, Heg}). Typically, this is achieved by considering sets $S$ and $T$ endowed with a particular structure. We focus on a recent application provided by Swaenepoel and Winterhof \cite{SwaWin}. In \cite[Theorem~1]{SwaWin}, the authors show the following: for a rational function $f\in \F_q(x)$ of degree $d$ satisfying \eqref{eqn:RNSHWINDC} 
so that  $f(T)\subseteq T$, there exists $U \subseteq T$ with $|U|\geq |T|/(d+1)$ such that
\begin{equation}
\label{eqn:SwaWinB}
    \bigg|\sum_{s\in S}\sum_{u\in U}\psi(su)\bigg|\ll \bigg(\frac{|S|^3|T|^3q}{M_T}\bigg)^{1/4},
\end{equation}
where $M_T$ is defined by \eqref{eqn:RNSHWIMExp}. See \cite[p.~3]{SwaWin} for a discussion of the strength of \eqref{eqn:SwaWinB}.

Following the work of Shkredov \cite{Shk}, we consider consequences of low-energy decomposition theorems to the finite field Littlewood problem of establishing non-trivial lower bounds on the $l_1$ norm of exponential sums over various sets.
See for instance \cite{Garc, Shk} for a background on this problem, and also for results showing that the $l_1$ norm of exponential sums, over images of intervals under various functions, is large. 

Finally, we mention that orbits of dynamical systems generated by functions $f\in \F_q(x)$ provide natural examples of sets $T\subset \F_q$, with $f(T)\subseteq T$. Namely, sets defined by 
\begin{equation}\label{eqn:OrbitsDef}
    \text{Orb}_f(u) = \{f^{(n)}(u):n\geq 0\},
\end{equation}
where $u\in \F_q$, $f^{(0)}(u) = u$ and $f^{(n)}(u) = f(f^{(n-1)}(u))$ for $n\geq 1$. Various arithmetical properties of such dynamical systems have been investigated, in particular, in \cite{Cha13, CGOS, GuiShp, Ost, RNShp}.

\subsection*{Notation}
For $p\neq 2$ prime and $q = p^n$ for $n\in \mathbb{N}$, $\mathbb{F}_q$ denotes the finite field of order $q$ and characteristic $p$. We write $\mathbb{F}_q^\times$ to denote the multiplicative group $\mathbb{F}_q\setminus\{0\}$.  

We write $\alpha\ll \beta$ or $\beta\gg \alpha$ if there exists an absolute constant $c>0$ such that $|\alpha|\leq c\beta$. If the constant $c$, depends on some parameter $\epsilon$, then we write, for example,  $\alpha\ll_{\epsilon} \beta$. If $\alpha\ll \beta$ and $\alpha\gg \beta$, we use $\alpha \approx \beta$. We also write $\alpha \lesssim \beta$ or $\beta \gtrsim \alpha$, if there exist $c_1, c_2>0$ such that $|\alpha| \leq c_1 (\log \beta)^{c_2}\beta$. If $\alpha\lesssim \beta$ and $\beta\lesssim \alpha$, we write $\alpha\sim \beta.$ We reserve the notation suppressing logarithmic factors for results pertaining to $\mathbb{F}_p$, where we do not expect our exponents to be optimal.  \\
For disjoint sets $S$ and $T$ we denote their union as $S\sqcup T$; a decomposition of $A$ into $S$ and $T$ exclusively refers to $A = S\sqcup T$ where $S$ and $T$ are disjoint.

\section{Main results}\label{sec:main results}
Our first result is for `small' sets $A$. 

\begin{theorem}\label{thm:bwdss} 
Let $A\subset \F_p$, with $|A|\leq p^{5/8}$and let $f\in \F_p\left[x, y\right]$ denote a non-degenerate quadratic polynomial. There exist disjoint subsets $S, T\subseteq A$ such that $A = S \sqcup T$ and
\begin{equation}
\label{eqn:BWDPolyD}
\max\{\E(S), \E_{f}(T)\}\lesssim |A|^{3-1/5}.
\end{equation}
\end{theorem}
We remark that this result automatically extends to sets and polynomials over arbitrary fields $\mathbb{F}$. In this setting, $A$ must satisfy the above size constraint in terms of $p$, the characteristic of $\mathbb{F}$; if the characteristic is zero, then there is no size constraint on $A$.

For large sets $A\subseteq \mathbb{F}_q$ our results are of a different flavour.
Instead of a decomposition of the set $A$, we obtain an energy-energy estimate for positive proportion subsets of $A$. Upon applications of the Cauchy-Schwarz inequality \eqref{eqn:AddEnCS}, this reproduces, up to logarithmic factors, Garaev's sum-product inequality  \cite[Theorem~1]{Gar}. 
\begin{theorem} 
\label{thm:LEDLL}
Let $A\subset \F_q$. There exist subsets $C\subseteq B\subseteq A$, with $|A|\ll |B| \ll (\log|A|)^{2} |C|$ such that
\[
\E(B)\E^{\times}(C) \ll \frac{|A|^7\log|A|}{q} + |A|^4(\log|A|)^2q\,.
\]
\end{theorem}
Theorem~\ref{thm:LEDLL} is non-trivial for $|A|\gtrsim q^{1/2}$ and as demonstrated by a construction in \cite{Gar}, it is optimal in the range $|A|\gtrsim q^{2/3}$. 

Finally, somewhat motivated by the low energy decomposition result of Roche-Newton et al. \cite[Theorem 1.1]{RNShWin}, we obtain the following energy-energy result:
\begin{theorem}
\label{thm:LEDRNSWI}
Let $A\subset \F_q$. There exist subsets $C\subseteq B\subseteq A$, with $|A|\ll |B| \ll (\log|A|)^{2} |C|$ such that
\[
\E(B)\E(f(C)) \ll \frac{|A|^7\log|A|}{q} + |A|^4(\log|A|)^2q
\]
\end{theorem}
Based on a construction in \cite[Section~1.3]{RNShWin}, Theorem~\ref{thm:LEDRNSWI} is sharp up to constants in the range $|A|> q^{2/3} (\log|A|)^{1/6}$, whereas \cite[Theorem 1.1]{RNShWin} is sharp in a range of the form $|A|\gg q^{9/13}(\log|A|)^{7/26}$.



\subsection{Applications}
Our first application is a quantitative improvement to an expansion result of Mirzaei \cite{Mirzaei}.
\begin{theorem}\label{thm:MirImp}
Let $A\subset \F_p$, with $|A|\ll p^{23/52}$ and let $f\in \F_p[x, y]$ be a non-degenerate quadratic polynomial. Then
\[
\max\{|A-A|, |f(A, A)|\}\gtrsim |A|^{28/23}.
\]
\end{theorem}
This improves the exponent of $6/5 + 4/305$ attained by Mirzaei to the exponent $6/5 + 2/115$. We note in particular that this result demonstrates the efficiencies that low-energy decomposition results yield. To see how low-energy decomposition results are used in the sum-product problem, we refer the reader to \cite{MohSte}.

As a second application, we give improvements of the main result of Swaenepoel and Winterhof \cite[Theorem~1]{SwaWin}:

\begin{theorem}
\label{thm:BCS1}
Let $S, T\subset \F_q$ and let $f\in \F_q(x)$ be a rational function of degree $d$ that satisfies \eqref{eqn:RNSHWINDC}. Suppose that $f(T)\subseteq T$. Then $T$ contains a large subset $U$ so that 
\[
|U|\gg \frac{|T|}{(d+1)(\log|T|)^2}
\] 
and
\begin{equation*}
    \bigg|\sum_{s\in S}\sum_{u\in U}\psi(su)\bigg|\ll \bigg(\frac{|S|^3|T|^3q}{M_T}\bigg)^{1/4},
\end{equation*}
where 
\begin{equation}\label{eqn:MALS}
M_T = \min\left\{\frac{q^{1/2}}{|T|^{1/2}(\log|T|)^{1/2}}, \frac{|T|}{q^{1/2}(\log|T|)}\right\}.
\end{equation}
\end{theorem}
This result is non-trivial in a range of the form $|S||T|^2 \gtrsim q^{3/2}$
and its strength increases as $|T|$ becomes larger than $|S|$;  
when $|S|=|T|$, this result is weaker than \eqref{eqn:VinoB}. 
Theorem~\ref{thm:BCS1} is a strict improvement over \cite[Theorem~1]{SwaWin}, as can be seen through a comparison of the quantities $M_T$ appearing in the two theorems (given by \eqref{eqn:MALS} and \eqref{eqn:RNSHWIMExp}). 
Furthermore, Theorem~\ref{thm:BCS1} yields quantitative improvements to \cite[Theorems~5 and 9]{SwaWin}.

As a third application, we provide a variant of Theorem~\ref{thm:BCS1} concerning small subsets of $\F_p$ and quadratic polynomials.
\begin{theorem}
\label{thm:BCS2}
Let $f\in \F_p\left[x\right]$ be a quadratic polynomial and let $T\subseteq \F_p$. Suppose that $f(T)\subseteq T$ and $ |T|\leq p^{5/8}$. Then there exists a subset $U\subseteq T$, with $|U| \gg |T|$ such that for any set $S\subseteq \F_p$, we have
\begin{equation}
\label{eqn:BCSDSB}
    \bigg|\sum_{s\in S}\sum_{u\in U} e_p(su)\bigg| \lesssim (|S|^3|T|^{3-\frac15} p)^{1/4}.
\end{equation}
Moreover, if $|S|\leq p^{5/8}$ and $f(S)\subseteq S$, then there exists $V\subseteq S$, with $|V|\gg |S|$, such that
\begin{equation}
\label{eqn:BCSDSB2}
    \bigg|\sum_{u\in U}\sum_{v\in V} e_p(uv)\bigg| \lesssim p^{1/8}(|S||T|)^{17/20}.
\end{equation}
\end{theorem}
To allow for a rough comparison between the estimates of Theorem~\ref{thm:BCS2} and Vinogradov's estimate \eqref{eqn:VinoB}, suppose we have sets $S, T, V, U\subset \F_p$ as given by Theorem~\ref{thm:BCS2}, with $|S|=|T|=N$ and let $f\in \F_p[x]$ be quadratic. Then 
\[
\bigg|\sum_{u\in U}\sum_{v\in V} e_p(uv)\bigg| \leq 
\begin{cases} 
Np^{\frac12} &\mbox{for } p^{\frac12+\frac1{18}}<N\leq p~; \\
N^{\frac{29}{20}}p^{\frac14+o(1)} & \mbox{for } p^\frac12<N\leq p^{\frac12+\frac1{18}}~;\\
N^\frac{17}{10}p^{\frac18+o(1)} & \mbox{for } p^{\frac12-\frac1{22}}<N\leq p^\frac12~;\\
N^2 & \mbox{for }  N<p^{\frac12-\frac1{22}}.
\end{cases}
\]
That is, in this demonstrative setting, Theorem~\ref{thm:BCS2} is superior to both Vinogradov's estimate and the trivial bound in the range $p^{\frac12-\frac1{22}}<N\leq p^{\frac12+\frac1{18}}$.

Finally, we turn our attention to the finite field Littlewood problem and prove lower bounds on the $l_1$ norm of exponential sums over certain types of sets.

\begin{theorem}\label{thm:LWP}
Let $f\in \F_p[x]$ denote a quadratic polynomial and let $A\subseteq \F_p$ be any set with $|A|\ll p^{2/3}$ and $|A+A|\ll |A|$, then 
\begin{equation}\label{eqn:WienerfI}
    \frac{1}{p}\sum_{\lambda \in\F_p}\left|\sum_{a\in f(A)}e_p(\lambda a)\right|\gtrsim |A|^{1/4}.
\end{equation} 
Let $T\subseteq \F_p$, with $|T|\ll p^{5/8}$ and suppose $f(T)\subseteq T$. Then
\begin{equation}\label{eqn:WienerOrbits}
\frac{1}{p}\sum_{\lambda \in\F_p}\left|\sum_{t\in T}e_p(\lambda t)\right|\gtrsim |T|^{1/10}.
\end{equation}
\end{theorem}
Estimate~\eqref{eqn:WienerfI} provides a quantitative improvement to a result of Garcia \cite[Corollaries~8 and 9]{Garc} for quadratic polynomials. Using Proposition~\ref{prop:bwd3}, one can recover, under a more favourable $p$-constraint on the set in question, the estimate of Shkredov \cite[Corollary~2]{Shk} concerning a lower bound on the $l_1$ norm of exponential sums over multiplicatively structured sets. 
We further note that estimate~\eqref{eqn:WienerOrbits} appears to be new in the sense that it provides a new class of examples toward the modular Littlewood problem. Finally, we mention that orbits of dynamical systems, defined in \eqref{eqn:OrbitsDef}, provide natural examples of sets $T$ to which Theorems~\ref{thm:BCS1}, \ref{thm:BCS2} and \ref{thm:LWP} apply.

\section{Preliminaries}
\subsection{Energy preliminaries}
We record the following energy sub-additivity lemma, which is a consequence of the Cauchy-Schwarz inequality.
\begin{lemma}
\label{lem:Efsubadd}
Let $f\in \F_q\left[x,y\right]$ and $V_1, \dots, V_k \subseteq \F_q$. Then
\[
E_{f}\bigg(\bigcup_{i=1}^{k}V_i\bigg) \leq \bigg(\sum_{i, j = 1}^k\E_{f}(V_i, V_j)^{1/2}\bigg)^2.
\]
Furthermore, if $f$ has the property that, for any $X, Y\subset \F_q$,
\begin{equation}
\label{eqn:EfCSCon}
   \E_{f}(X, Y) \ll\E_{f}(X)^{1/2}\E_{f}(Y)^{1/2},
\end{equation}
then we have
 \begin{equation}
 \label{eqn:EfSubaddCS2}
    \E_{f}\bigg(\bigcup_{i=1}^{k}V_i\bigg) \ll \bigg(\sum_{i=1}^k\E_{f}(V_i)^{1/4}\bigg)^4.
 \end{equation}
 \end{lemma}
   \begin{proof}

 Without loss of generality, we may assume that the sets $V_i$, $1\leq i\leq k$ are pairwise disjoint. Thus, by an application of the Cauchy-Schwarz inequality, we have
 \begin{align*}
     \E_{f}\bigg(\bigcup_{i=1}^{k}V_i\bigg) &= \sum_{i, j, k, l = 1}^k\sum_{\lambda\in \F_q}r_{f(V_i, V_j)}(\lambda)\cdot r_{f(V_k, V_l)}(\lambda) \\
     &\leq \sum_{i, j, k, l = 1}^k\bigg(\sum_{\lambda\in \F_q}r_{f(V_i, V_j)}(\lambda)^2\bigg)^{1/2} \bigg(\sum_{\lambda\in \F_q}r_{f(V_k, V_l)}(\lambda)^2\bigg)^{1/2} 
     = \bigg(\sum_{i, j = 1}^k\E_{f}(V_i, V_j)^{1/2}\bigg)^2.
     \end{align*}
     Applying \eqref{eqn:EfCSCon} to the above inequality gives \eqref{eqn:EfSubaddCS2}.
 \end{proof}

Key to our results is the following regularisation lemma as recorded by Xue \cite{Xue}. Although Xue formulates the regularisation in $\mathbb{R}$, the proof is valid over abelian groups. This regularisation has origins in \cite[Proposition 16]{RudShSt}.

\begin{lemma}
\label{lem:regularization}
Let $A$ be a subset of an abelian group. There exist subsets $C\subseteq B\subseteq A$, with $|A|\ll |B| \ll (\log|A|)^{2} |C|$, a number $1\leq t\leq |B|$ and a set $D = \{x\in B - B:t\leq r_{B-B}(x)<2t\}$ such that 
\[
|D|t^2 \ll \E(B) \ll (\log|B|)|D|t^2
\]
and for any $c\in C$, 
\[
r_{D + B}(c) \gg \frac{|D|t}{|B|}\,.
\]
\end{lemma}

\subsection{Energy estimates and incidence theorems}
A typical application of incidence geometry is to energy estimates.

Using Rudnev's \cite{Rud18} incidence theorem between points and planes in $\mathbb{F}^3$, Koh, Mirzaei, Pham and Shen \cite[Theorem~2.1]{KoMiPhSh} obtained the following energy estimate, generalising a result of Pham, Vinh and de Zeeuw \cite{PhaVindeZ}.
\begin{lemma}
\label{lem:KoMiPhSh3Pol}
Given sets $U, V, W\subseteq \F_p^{\times}$, with $|U||V||W|\ll p^2$ and a quadratic polynomial $f\in \F_p\left[x, y, z\right]$, which depends on each variable and is not of the form $g(h(x)+k(y) + l(z))$, we have
\begin{align*}
    |\{(u_1, u_2, v_1, v_2, w_1, w_2)&\in U^2 \times V^2 \times W^2: f(u_1, v_1, w_1) = f(u_2, v_2, w_2)\}| \\ &\ll (|U||V||W|)^{3/2} + \max\{|U|^2|V|^2, |U|^2|W|^2, |V|^2|W|^2\}.
\end{align*}
\end{lemma}

For large sets we have a similar result, relying instead on point-line incidence estimates due to Vinh \cite{Vinh}:
\begin{lemma}
\label{lem:VinhPLI} 
Let $P$ denote a set of points and $L$ a collection of lines over $\F_q^2$. 
Then 
$$
\bigg ||\{(p, l)\in P\times L: p\in l\}|- \frac{|P||L|}{q}\bigg|\leq \sqrt{q|P||L|}.
$$
\end{lemma}

Vinh's bound enables us to obtain a bound on the multiplicative representation function which we recycle into an energy estimate using standard techniques.
\begin{lemma}
\label{lem:SUB4dyadicMBW}
Let $A,B,C \subseteq \mathbb{F}_q^\times$ and  suppose that there are sets $Q,R \subseteq \mathbb{F}_q$ and a number $T\geq 1$ so that $r_{Q+R}(a)\geq T$ for each $a\in A$.
Then 
\[
\E^\times(A,B)  \ll \frac{|A||B|^2|Q||R|}{Tq} + \frac{q|Q||R||B|\log_2(|A|)}{T^2}\,.
\]
\end{lemma}
\begin{proof}
Let us fix $\tau = \frac{4|Q||R||B|}{Tq}$ and write
$X_i = \{x\in AB\colon \tau 2^i < r_{AB}(x)\leq  \tau 2^{i+1}\}$. Then 
\begin{align*}
\E^\times(A,B) &= \sum_{x: r_{AB}(x) \leq\tau} r_{AB}^2(x) + \sum_{i\geq 0}\sum_{x\in X_i}r_{AB}(x)^2 \leq |A||B|\tau + \sum_{i\geq 0}|X_i|\tau^2 2^{2i+2}\,.
\end{align*}

We estimate $|X_i|$ using Lemma~\ref{lem:VinhPLI}:
\begin{align*}
\tau 2^i|X_i| &\leq |\{(x,a,b)\in X_i\times A\times B\colon x = ab\}| \leq \frac1{T}|\{(x,q,r,b)\in X_i\times Q\times R \times B \colon x = (q+r)b\}| \\
&\leq \frac{|X_i||Q||R||B|}{qT} + \frac{\sqrt{q |X_i||Q||R||B|}}{T}\,. 
\end{align*}

If the first term dominates, then we obtain $\tau \leq 2 \frac{|Q||R||B|}{qT2^i}$; our choice of $\tau$ yields a contradiction and so we have the estimate
\[
|X_i| \tau^2 \leq \frac{q|Q||R||B|}{T^2 2^{2i-2}}\,.
\]
Finally, plugging in these bounds into the above expression for $\E^\times(A,B)$ completes the proof.  
\end{proof}

We record an analogous result where Vinh's incidence estimate is replaced by a lemma of Roche-Newton, Shparlinski and Winterhof \cite[Lemma~2.2]{RNShWin}, which relies instead on Weil bounds.

\begin{lemma}\label{lem:RNShWin}
Let $A,B, C \subseteq \mathbb{F}_q^\times$ and suppose that there are sets $Q,R \subseteq \mathbb{F}_q$ and a number $T\geq 1$ so that $r_{Q+R}(c)\geq T$ for each $c \in C$. Let $f\in \F_q(x)$ denote a  rational function of degree $d$ that is non-degenerate in the sense of \eqref{eqn:RNSHWINDC}.
Then 
\[
\E(f(A), B)  \ll_d \frac{|A||B|^2|Q||R|}{Tq} + \frac{q|Q||R||B|\log_2(|A|)}{T^2}\,.
\]
\end{lemma}

\section{Proof of Theorem~\ref{thm:bwdss}}
The proof of Theorem~\ref{thm:bwdss} proceeds in three stages. First we show how to obtain a subset with advantageous additive structure. Then we demonstrate how this additive structure enables a suitable energy bound on $\E_f$. Finally we apply an algorithmic procedure of Balog and Wooley \cite{BW} to provide a decomposition of $A$.

\subsection{Finding a subset with additive structure}
The procedure to find a subset with additive structure is an involved pigeonholing argument with geometry in the background. This lemma can also be extracted from the proof of ~\cite[Proposition~3.1]{RudShSt}.
\begin{lemma}
\label{lem:EnergyEnergyPigeonh}
Let $X\subset \F$. Then:
\begin{enumerate}[label=(\roman*)]
    \item there exist sets $D \subset X+X$ and $1\leq \tau \leq |X|$ so that 
\begin{equation*}
    \E(X)\gg \frac{|D|\tau^2}{\log|X|} \quad \text{and}\quad r_{X+X}(d)\in [\tau, 2\tau) \quad \text{for all}\quad d\in D\,.
\end{equation*}
 \item there exists $X_*\subseteq X$ so that 
\begin{equation*}
    \label{eqn:A*LNEnergy}
    |X_*|^2 \gg \frac{\E(X)}{|X| (\log|X|)^{7/2}} \quad \text{and} \quad 
    r_{D-X}(x) \geq \frac{|D|\tau}{|X_*|(\log|X|)^2} \quad \text{for all} \quad x\in X_*.
\end{equation*}
\end{enumerate}
\end{lemma}

\begin{proof}
Let us write $D_i=\{d\in X + X: r_{X+X}(d)\in [2^i,2^{i+1})\}$ for $i = 0,\dots, \lceil\log_2|X|\rceil$. Writing
\[
\E(X) = \sum_{i\geq 0}\sum_{d\in D_i}r_{X+X}^2(d)\,,
\]
a pigeonholing argument shows that there exists $i_0$ so that, setting $D = D_{i_0}$ and $\tau = 2^{i_0}$, we have  
\begin{equation}
\label{eqn:Energytau2D}
\frac{\E(X)}{\log|X|}\ll \tau^{2}|D| \ll \E(X)\,.
\end{equation}

We now extract the subset $X_*\subseteq X$ in a geometrically motivated manner. 
Let $P_1 = \{(x, y) \in X\times X: x+y\in D\}$ be the set of points in $X\times X$ lying on the line with slope $-1$ and with a vertical axis-intercept in $D$.
For $x\in X$, let $A_x = \{y: (x, y) \in P_1\}$ be the set of ordinates of $P_1$ and note that 
\begin{equation}
\label{eqn:P1Size}
\sum_{x\in X}|A_x| = |P_1| \quad \text{and}\quad \tau|D| \leq |P_1| < 2\tau|D|.
\end{equation}

By a dyadic pigeonholing argument, we extract a set of ``popular abscissae": 
let $V_i =\{x\in X: 2^i \leq |A_x |< 2^{i+1}\}$ for $i = 0,\dots, \lceil \log_2|A|\rceil$. Then 
\[
|P_1| = \sum_{i\geq 0}\sum_{v\in V_i}|A_v|
\]
and so by a pigeonholing argument there exists $i_0$ so that, setting $V=V_{i_0}$ and $\kappa_1 = 2^{i_0}$ we have
\begin{equation}
\label{eqn:VKappa1LB1}
    |V|\kappa_1 \gg \frac{|P_1|}{\log |X|} \gg \frac{\tau|D|}{\log |X|}.
\end{equation}

Note that for $v\in V$, we have $|A_v|\in [\kappa_1,2\kappa_1)$. Geometrically, this corresponds to at least $\kappa_1$ intersections between $P_1$ and the vertical line through $(v,0)$. Thus $r_{D-X}(v)\geq \kappa_1$ for each $v\in V$.

We now split into two cases.

{\bf Case 1:} $|V|\geq \kappa_1 (\log|X|)^{-1/2}$.\\ 
In this case, we take $ X_* = V$ and $\kappa = \kappa_1$. We have $r_{D-X}(x)\geq \kappa$ for each $x\in X_*$. Moreover, from \eqref{eqn:VKappa1LB1}, it follows that $\kappa \gg \tau |D|(\log|X||X_*|)^{-1}$.
We have 
\[
|X_*|\geq \kappa (\log|X|)^{-1/2} \gg \frac{\tau |D|}{(\log|X|)^{3/2}|X_*|} = \frac{\tau^2 |D|}{(\log|X|)^{3/2}|X_*|\tau}\gg \frac{\E(X)}{(\log|X|)^{5/2} |X_*|}\,.
\]

This concludes the proof of the theorem in Case 1. 

{\bf Case 2:} $|V|< \kappa_1 (\log|X|)^{-1/2}.$\\
We perform another pigeonholing argument to yield a set of ``popular ordinates":
Define $P_2 = \{(x, y) \in P_1: x\in V\}$. By definition of $V$, we have $\kappa_1\leq |A_x|<2 \kappa_1$  for each $x\in V$ and so
\begin{equation*}
    \label{eqn:P2Size}
    |V|\kappa_1\leq |P_2| < 2 |V|\kappa_1\,.
\end{equation*}
For $y\in X$, define the set of ordinates $B_y = \{x: (x, y)\in P_2\}$. Note that
$
\sum_{y\in X}|B_y| = |P_2|. 
$

Arguing by a dyadic pigeonhole argument as before, we obtain $1\leq \kappa_2\leq |V|$ and $U = \{y\in X: \kappa_2\leq |B_y|<2\kappa_2\}$ so that (using \eqref{eqn:VKappa1LB1})
\begin{equation}
\label{eqn:Ukappa2LB1}
|U|\kappa_2 \gg \frac{|P_2|}{\log |X|} \geq \frac{|V|\kappa_1}{\log |X|} \gg \frac{\tau |D|}{(\log|X|)^2}\,.
\end{equation}

We can interpret the set $U$ geometrically as follows: for any $u\in U$, the horizontal line through $(0,u)$ intersects at least $\kappa_2$ points of $P_2$. Hence $r_{D-X}(u) \geq \kappa_2$ for each $u\in U$.

Suppose to the contrary that $|U|\ll \kappa_2(\log|X|)^{-1/2}$. Then 
$$|U|\kappa_2 \ll \kappa_2^2(\log|X|)^{-1/2} \leq |V|^2(\log|X|)^{-1/2} < |V|\kappa_1 (\log|X)^{-1}\,.$$ 
For a suitable choice of constants, this contradicts \eqref{eqn:Ukappa2LB1}. Hence we assume that $|U|\gg \kappa_2(\log|X|)^{-1/2}$.

Set $X_* = U$ and $\kappa = \kappa_2$. By \eqref{eqn:Ukappa2LB1} we have $r_{D-X}(x)\geq \tau |D|(|X_*|(\log|X|)^2)^{-1} $ for all $x\in X_*$. 

Furthermore, 
\[
|X_*|\gg \kappa (\log|X|)^{-1/2} \gg \frac{\tau |D|}{|X_*|(\log|X|)^{5/2}}= \frac{\tau^2 |D|}{\tau|X_*|(\log|X|)^{5/2}}\gg \frac{\E(X)}{|X||X_*|(\log|X|)^{7/2}}\,.
\]
\end{proof}

\subsection{Mixed energy bounds}
We will show how the advantageous additive structure from the subset derived in the previous subsection is amenable to mixed energy bounds. This is a preliminary decomposition result towards Theorem~\ref{thm:bwdss}.

We will apply Lemma \ref{lem:KoMiPhSh3Pol} and so require the following claim, the justification of which is provided within the proof of \cite[Theorem~1.10]{KoNassPhamVal}.

\begin{claim}
\label{claim:nondegbi23}
Let $f\in \F_q\left[x, y\right]$ denote a non-degenerate quadratic polynomial. Define $f^{'}\in \F_q\left[x, y, z\right]$ by $f^{'}(x, y, z) = f(x\pm z, y)$. Then $f^{'}$ is not of the form $g(h(x)+k(y) + l(z))$ for any univariate polynomials $g, h, k, l$ over $\F_q$. 
\end{claim}

We now state and prove our mixed energy bound.

\begin{proposition}\label{prop:bwd3}
Let $f\in \F_p\left[x, y\right]$ be a non-degenerate quadratic polynomial and let $A\subseteq \F_p$. Suppose $X\subseteq A$ satisfies the conditions
\begin{equation}
\label{eqn:Bilemcharcon}
    |X|^{5}|A| \leq p^{2}\E(X)
\end{equation}
and
\begin{equation}
\label{eqn:BienergyLBCon}
\E(X)\gg |A|^{3-1/3}.
\end{equation}
Then there exists a set $X_* \subseteq X$, with $|X_*|^2\gtrsim \E(X)|X|^{-1}$ such that for any $Y\subseteq A$, with $|Y|\geq |X_*|$, we have 
\begin{equation}
\label{eqn:lembiEX*UB}
    \max\{\E_{f}(X_*, Y), \E_{f}(Y, X_*)\} \lesssim \frac{|X_*|^{4}|Y|^{3/2}|A|^{3/2}}{\E(X)^{3/2}}.
\end{equation}
\end{proposition}
\begin{proof}
We apply Lemma~\ref{lem:EnergyEnergyPigeonh} to the set $X$ and henceforth assume its full statement and notation. 

By the additive structure of $X_*$, namely that for each $x\in X_*$, $r_{D-X}(x) \gtrsim |D|\tau |X_*|^{-1}$, we have
\begin{align}
\label{eqn:EfUBE}
\nonumber    \E_f(X_*, Y) &= |\{(x_1, x_2, y_1, y_2)\in X_*^2\times Y^2: f(x_1, y_1) = f(x_2, y_2)\}|\\
    &\lesssim \frac{|X_*|^2}{|D|^2\tau^2}|\{(d_1, d_2, x_1, x_2, y_1, y_2)\in D^2\times X^2\times Y^2: f(d_1-x_1, y_1) = f(d_2-x_2, y_2)\}|.
\end{align}

Define $f^{'}(u, v, w) = f(u-v, w)$. By Claim~\ref{claim:nondegbi23}, it follows that $f^{'}$ is not of the form $g(h(x)+k(y) + l(z))$. Using Lemma~\ref{lem:KoMiPhSh3Pol}, we will obtain an upper bound on the quantity
\begin{equation}
\label{eqn:3vpolyEdefn}
E := |\{(d_1, d_2, x_1, x_2, y_1, y_2)\in D^2\times X^2\times Y^2:f^{'}(d_1, x_1, y_1) = f^{'}(d_2, x_2, y_2)\}|\,.
\end{equation}

We justify the $p$-constraint necessary to the application of Lemma~\ref{lem:KoMiPhSh3Pol}: we require that $|D||X||Y|\ll p^2$. Note that 
$|X_*|^2\tau \geq \tau^2|D|\gg \E(X)(\log|X|)^{-1}$ and so $\tau \gg \E(X)(|X_*|^2\log|X|)^{-1}$; hence, since $|D|\tau^2 \leq \E(X)$ we have the upper bound
\begin{equation}
    \label{eqn:UBD}
    |D|\ll |X_*|^4(\log|X|)^2\E(X)^{-1}. 
\end{equation}
Thus, $|D||X||Y|\ll (\log|X|)^2|X_*|^4|X||A|(\E(X))^{-1} \leq p^2(\log|X|)^2$. By the assumption \eqref{eqn:Bilemcharcon}, the required bound holds up to a logarithmic factor; to drop this factor, we replace $X_*$ with any of its subsets of size $\geq |X_*|(\log|X|)^{-1}$, without changing notation.
This affects only the logarithmic factor in the lower bound on $|X_*|$, of which we do not keep track.

By Lemma~\ref{lem:KoMiPhSh3Pol}, we have
\begin{align}
\label{eqn:EUBWMax}
     E &\ll (|D||X||Y|)^{3/2} + \max\{|D|^2|X|^2, |D|^2|Y|^2, |X|^2|Y|^2\}\nonumber\\
     &\ll (|D||A||Y|)^{3/2} + \max\{|D|^2|A|^2, |A|^2|Y|^2\}.
\end{align}
We will show that, up to a logarithmic factor, the first term of \eqref{eqn:EUBWMax} dominates. We split into two cases depending on $M:= \max\{|D|^2|A|^2, |D|^2|Y|^2, |A|^2|Y|^2\}$.

{\bf Case 1:} $M=|D|^2|A|^2$.
It suffices to show $|A||D| \lesssim |Y|^3$. Applying \eqref{eqn:UBD} and the assumptions~\eqref{eqn:BienergyLBCon} (which gives $\E(X)\gg |A|^{8/3}$) and $|X_*|\leq |Y|\leq |A|$ yields
\begin{align*}
    |A||D| \lesssim \frac{|A||X_*|^4}{\E(X)} \leq \frac{|A||X_*|^4}{|A|^{8/3}} \leq \frac{|Y|^3|X_*|}{|A|^{5/3}} \leq |Y|^3\,.
\end{align*}


{\bf Case 2:} $M=|X|^2|Y|^2$.
Here, we wish to show $|X||Y|\lesssim |D|^3$. Recalling the lower bound $|D|\gtrsim \E(X)|X|^{-2}$ and the assumption~\eqref{eqn:BienergyLBCon}, we have
\[
|D|^3\gtrsim \frac{\E(X)^3}{|X|^6}\geq \frac{|A|^8}{|X|^6} \geq |A|^2\geq |X||Y|\,.
\]

Finally, with the assumption that $E\lesssim (|D||A||Y|)^{3/2}$, we demonstrate the bound \eqref{eqn:lembiEX*UB}: from \eqref{eqn:EfUBE} we have
\begin{align*}
\E(X_*, Y) &\lesssim \bigg(\frac{|X_*|}{|D|\tau}\bigg)^2\cdot \Big(|D||A||Y|\Big)^{3/2} = \frac{|X_*|^{2}|A|^{3/2}|Y|^{3/2}}{|D|^{1/2}\tau^2}\cdot \frac{|D|\tau}{|D|\tau}\\
&\lesssim \frac{|X_*|^4|A|^{3/2}|Y|^{3/2}}{\E(X)^{3/2}}.
\end{align*}
Here we use the upper bound $|D|\tau \lesssim |X_*|^2$ implicit in the proof of Lemma~\ref{lem:EnergyEnergyPigeonh}. This scheme gives the same bound for $\E(Y,X_*)$.
\end{proof}

\subsection{Proof of Theorem~\ref{thm:bwdss}} 
To prove the strengthed decomposition of Theorem~\ref{thm:bwdss} we will use the following decomposition algorithm from Balog and Wooley \cite{BW}. At a high level, it involves iteratively removing subsets $B\subseteq A$ with large $\E(B)$ (as determined by a parameter $M$).

\begin{algorithm}[H]\label{algo}
\SetAlgoLined
\SetKwInOut{Input}{Input}\SetKwInOut{Output}{Output}

\Input{$A\subseteq \mathbb{F}$, suitably chosen $M \in [1,|A|] $}
\Output{Decomposition $A = S \sqcup T$}
\BlankLine
 Initialisation\: $S_0 = A$, $T_0 = \emptyset$\;
 \While{$\E(S_i) > |A|^3 M^{-1}$}{
    extract $B_i \subseteq S_i$ using Proposition~\ref{prop:bwd3} \;
    $S_{i+1} = S_i \setminus B_i$\;
    $T_{i+1} = T_i \sqcup B_i$\;
   }
 \caption{Decomposition Algorithm}
\end{algorithm}
When we apply Proposition~\ref{prop:bwd3} in the algorithm, we take $X = S_i$ and extract $X_* = B_i$.
Hence for any set $Y\subseteq A$ satisfying $|Y|\geq |B_i|$:
\[
\max\bigg\{\E_f(B_i, Y), \E_f(Y,B_i)\bigg\}^2 \lesssim \frac{M^3 |B_i|^8|Y|^3}{|A|^6}\,.
\]

At each stage $i$ of the algorithm we have $A = S_i \sqcup T_i$. By the uniform lower bound $|B_i|^2\geq |A|^2 M^{-1}$, the size of $S_i$ is uniformly decreasing and so the algorithm terminates.   
Suppose that the algorithm terminates at the $k$th iteration. That is, $A = S_k \sqcup T_k$ and $\E(S_k) \leq |A|^3 M^{-1}$. We aim to determine the parameter $M$ so that we also have $\E_f(T_k) \leq |A|^3 M^{-1}$. 

Let us reorder the $B_i$ so that $|B_1|\leq|B_2|\leq \dots$. Then 
\begin{align*}
    \E_f(T_k) &= \E_f\bigg(\bigsqcup_{i = 0}^k B_i \bigg) \leq \bigg( \sum_{i, j = 1}^k\E_{f}(B_i, B_j)^{1/2}\bigg)^2\\
    &= \bigg(\sum_{i=1}^k\sum_{j=1}^i \E_f(B_i,B_j)^{1/2} + \sum_{i=1}^k \sum_{j = i+1}^k \E_f(B_i,B_j)^{1/2}
    \bigg)^2\\ 
    & \lesssim \bigg(\sum_{i=1}^k \sum_{j = 1}^i \frac{M^{3/4} |B_j|^2|B_i|^{3/4}}{|A|^{3/2}}
    +
    \sum_{i = 1}^k\sum_{j=i+1}^k \frac{M^{3/4} |B_i|^2|B_j|^{3/4}}{|A|^{3/2}}
    \bigg)^2\\ 
    & \leq \frac{M^{3/2}}{|A|^{3}}\bigg(\sum_{i = 1}^k\sum_{j=i+1}^k |B_i||B_j||A|^{3/4}
    +\sum_{i=1}^k \sum_{j = 1}^i |B_j||B_i||A|^{3/4} 
    \bigg)^2\\
    & \ll M^{3/2}|A|^{5/2}\,. 
\end{align*}

Finally we optimise $M$: choosing $M = |A|^{1/5}$ yields $\max\{\E(S), \E_f(T)\}\lesssim |A|^{3 - \frac15}$\,.
 
It remains to justify the application of Proposition~\ref{prop:bwd3} with this choice of $M$. Each application of Proposition~\ref{prop:bwd3} proceeds under the assumption that $\E(S_i)>|A|^{3 -\frac15}$ and so clearly the condition \eqref{eqn:BienergyLBCon} is satisfied.

Let us now justify the $p$-constraint. Suppose for contradiction that, before terminating, there is $i$ so that $|S_i|^5|A|> p^2 \E(S_i)$. Then, recalling that $|A|\leq p^{5/8}$, we conclude that
\[
|S_i|^5 |A|> \bigg(|A|^\frac85\bigg)^2|A|^{3-\frac15} \implies |S_i| > |A|\,,
\]
which is a contradiction. This concludes the proof of Theorem~\ref{thm:bwdss}.

\section{Proofs of  Theorem~\ref{thm:LEDLL} and Theorem~\ref{thm:LEDRNSWI}}
Since both proofs are almost identical, we prove only  Theorem~\ref{thm:LEDLL}. To prove  Theorem~\ref{thm:LEDRNSWI}, it suffices to replace Lemma~\ref{lem:SUB4dyadicMBW} by Lemma~\ref{lem:RNShWin}.

Let $B, C, D$ be the sets given by Lemma~\ref{lem:regularization} so that $\E(B) \ll(\log|A|) |D|t^2$ for some $t\geq 1$.  We have
\[
r_{D+B}(c)\gg \frac{|D|t}{|B|}.
\]
Hence by Lemma~\ref{lem:SUB4dyadicMBW} we obtain
\begin{align*}
\label{eqn:ECXitorho}
\E^{\times}(C) &\ll \frac{|C|^3|D||B|^2}{ |D|tq} + \frac{q|D||B|^3|C|\log_2(|A|)}{|D|^2t^2} \\
&\leq \frac{|C|^3|B|^4}{ |D|t^2q} + \frac{q|B|^3|C|\log|A|}{|D|t^2} \\
&\leq \frac{|A|^7 \log|A|}{ \E(B)q} + \frac{q|A|^4(\log|A|)^2}{\E(B)}\,.
\end{align*}

\section{Proof of Theorem~\ref{thm:MirImp}} 
The strategy for proving Theorem~\ref{thm:MirImp} is as follows. First, we pass to subsets $C\subseteq B\subseteq A$ which enable `good' mixed energy bounds in terms of $\E(B,\cdot)$ and  $\E_f(C,\cdot)$. Then we amplify these mixed energy bounds by turning to the fourth moment energy using arguments of Rudnev, Shakan and Shkredov \cite{RudShaShk}. Finally, a result of Mirzaei \cite{Mirzaei} enables a good estimate on the fourth moment energy.

\subsection{Subsets and mixed energy bounds}
We require the following variant of Lemma~\ref{lem:regularization} which may be deduced from \cite[Proposition 1]{StWa}.
\begin{lemma}
\label{lem:regularization2}
Let $A$ and $V$ be finite subsets of an abelian group. There exist subsets $C\subseteq B\subseteq A$, with $|A|\ll |B| \ll (\log|A|)^{2} |C|$, a number $1\leq t\leq |B|$ and a set $D = \{x\in B - V:t\leq r_{B-V}(x)<2t\}$ such that 
\[
|D|t^2 \ll \E(B, V) \ll (\log|A|)|D|t^2
\]
and for any $c\in C$, 
\[
r_{D + V}(c) \gg \frac{|D|t}{|A|}\,. 
\]
\end{lemma}

\begin{proposition}\label{prop:bwdss2}
Let $A, V, X\subseteq \F_p$, with $|V|, |X|\gg |A|$. Let $f \in \F_p[x,y]$ denote a non-degenerate quadratic polynomial. Then there exist subsets $C\subseteq B \subseteq A$ with $|C| \gtrsim |B| \gg |A|$ so that, if $|X||B|^2|V|^3/ \E(B, V)\ll p^2$, then 
\[
\E(B, V)^{3} \E_f(C, X)^2 \lesssim |A|^6|X|^3|V|^5.
\]
\end{proposition}
\begin{proof}

We first apply Lemma~\ref{lem:regularization2}, to get subsets $C\subseteq B \subseteq A$ and a set $D\subseteq B-V$ so that $r_{B-V}(d)\in [t,2t)$ for each $d\in D$ and $\E(B, V)\gtrsim |D|t^2$. Moreover, $r_{D+V}(c)\gg |D|t|A|^{-1}$ for all $c\in C$.

Without loss of generality, we may assume that $0\notin D$, removing it if necessary. Indeed, if $|D|\ll 1$, then $\E(B, V)\ll |B||V|$ and so we are done using the trivial bound $\E_f(C,X)\leq |C|^2|X|$. Otherwise, if $|D|\gg 1$, then $\E(B, V)\gg (|D|-1)t^2$, and so removing an element from $D$ is without consequence.

By Claim~\ref{claim:nondegbi23}, we see that the polynomial $\tilde f \in \mathbb{F}_p[x,y,z]$ defined by $\tilde f(x,y,z) = f(x+y,z)$ is not of the form $g(h(x) + k(y) + l(z))$. 

Let us now apply Lemma~\ref{lem:KoMiPhSh3Pol}, deferring the justification of the $p$-constraint until the end of the proof. We have
\begin{align*}
\E_f(C,X) &:= |\{(c_1,c_2,x_1,x_2)\in C^2 \times X^2 \colon f(c_1,x_1) = f(c_2,x_2)\}|\\
&\lesssim \frac{|A|^2}{|D|^2t^2}
|\{(v_1,v_2,d_1,d_2,x_1,x_2)\in V^2\times D^2 \times X^2 \colon f(v_1 + d_1,x_1) = f(v_2 + d_2,x_2)\}|\\
&\lesssim \frac{|A|^2}{|D|^2t^2}\max\left\{|D|^{3/2}|X|^{3/2}|V|^{3/2}, |V|^2|X|^2 , |V|^2|D|^2 ,|X|^2|D|^2 \right\}\\
&:= \frac{|A|^2}{|D|^2t^2}M\,.
\end{align*}

We split into cases according to the maximand in the above expression. We shall in particular make repeated use of the trivial bound $\E_f(C, X)\leq |C||X|\cdot\min\{|C|, |X|\}.$

{\bf Case 1: $M = |D|^{3/2}|X|^{3/2}|V|^{3/2}$.} Squaring and multiplying by $(|D|t)^2$ we have
\[
(|D|t^2)^3 \E_f(C, X)^2 \lesssim |A|^4|X|^3|V|^3 (|D|t)^2.
\]
Using $|D|t \leq |B||V|$ completes the proof in this case.

{\bf Case 2: $M = |V|^2 |X|^2$.} In particular, we have that $|D|^3 < |V||X|$, and so, using the trivial bound $t\leq |B|$, we have $\E(B, V)^3\E_f(C,X)^2 < |V||X| \cdot |B|^6 \cdot (|X||C|^2)^2 \leq |A|^{10}|X|^3|V|$.

{\bf Case 3: $M =|V|^2|D|^2$.} We have $|X|^3 <|V||D|$. Thus
\begin{align*}
    \E(B, V)^3\E_f(C, X)^2 &\leq (|D|t)^3\cdot t \cdot (t^2\E_f(C, X))\cdot (|C|^2|X|) \lesssim (|B||V|)^3\cdot |B|\cdot (|A|^2|V|^2)\cdot (|C|^2|X|)\\
    &\leq |A|^8|X||V|^5\ll |A|^6|X|^3|V|^5.
\end{align*}

{\bf Case 4: $M = |X|^2|D|^2$.} We get
\begin{align*}
|D|\E(B,V)^3\E_f(C,X)^2 & \sim (t|D|)^4 (t^2~\E_f(C,X)) \E_f(C,X) \lesssim |B|^4|V|^4 (|A|^2|X|^2) |C|^2|X| \\
&\ll |A|^7|X|^3 |V|^4 |B|\leq|A|^7|X|^3 |V|^4|D|\,. 
\end{align*}
Thus 
\[
\E(B,V)^3\E_f(C,X)^2 \lesssim |A|^7|X|^3 |V|^4\leq |A|^6|X|^3|V|^5\,. 
\]

Finally, let us justify the $p$-constraint necessary for our application of Lemma~\ref{lem:KoMiPhSh3Pol}. Note that by Lemma~\ref{lem:regularization2}, we have
\[
\E(B, V) |D| \lesssim (|D|t)^2\ll |B|^2|V|^2.
\]
Hence if $|X||B|^2|V|^3/ \E(B, V)\ll p^2$, our use of Lemma~\ref{lem:KoMiPhSh3Pol} is justified.
\end{proof}

\subsection{Fourth moment energy}
The following lemma is proved by Mirzaei \cite{Mirzaei}.
\begin{lemma}
\label{lem:MirE4}
Let $A, B\subset \F_p$, with $|A||B||A-B|\ll p^2$. For a non-degenerate, quadratic polynomial $f\in \F_p[x, y]$, we have
\[
\E_4(A, B)\ll \frac{|f(A,A)|^2|B|^3}{|A|}.
\]
\end{lemma}
We extract the following from the arguments of Rudnev, Shakan and Shkredov \cite[Equations~3.2 and 3.3]{RudShaShk}.
\begin{lemma}\label{lem:RSSDC}
Let $A\subset \F_p$, then
\[
|A|^{24}\leq \E_4(A)^2|A-A|^{5}\Delta^4\E_4(A,D),
\]
for some $\Delta\geq 1$ and $D\subset A- (A-A)$ such that $r_{A-(A-A)}(d)\approx \Delta$ for all $d\in D$.
\end{lemma}

\subsection{Concluding the proof of Theorem~\ref{thm:MirImp}}
We recall a Pl\"unnecke-Ruzsa type inequality appearing in \cite{Pet12}.
\begin{lemma}\label{lem:PRI1}
Given finite, non-empty sets $A$ and $B$ in an abelian group, we have
\[
|kA - lA| \leq \frac{|A+B|^{k+l}}{|B|^{k+l-1}},
\]
where $kA$ is used to denote the $k$-fold sum set of $A$.
\end{lemma}

Let us now prove Theorem~\ref{thm:MirImp}, deferring justifications of $p$-constraints until after we have obtained the result. 

By Proposition~\ref{prop:bwdss2}, there exist $C\subseteq B\subseteq A$, with $|C|\gtrsim |B|\gg |A|$ such that
\begin{equation}\label{eqn:UBEMI}
\E(B, A-A)\lesssim |A|^{1/3}|f(C, C)|^{2/3}|A-A|^{5/3}.
\end{equation}

Next, applying Lemma~\ref{lem:MirE4} to Lemma~\ref{lem:RSSDC}, under the condition $|A||A+A-A||A+A-A-A|\ll p^{2}$, we obtain
\begin{align*}
    |A|^{25}\ll& \E_4(B)^2|B-B|^{5}|f(B, B)|^2(|D|\Delta)^2 (|D|\Delta^2)\\
    &\ll \E_4(B)^2|B-B|^{5}|f(B, B)|^2(|B||B-B|)^2 (\E(B, B-B)).
\end{align*}

A second application of Lemma~\ref{lem:MirE4} gives
\begin{equation}
\label{eqn:RSSDC1}
|A|^{19} \lesssim |f(A, A)|^6|A-A|^7 \E(B, A-A).
\end{equation}
Applying \eqref{eqn:UBEMI} to this gives the result. 

Let us now justify the $p$-constraints. First, to justify the application of Proposition~\ref{prop:bwdss2}, we note that if the relevant $p$-constraint fails, we get
\[
\frac{|A-A|^3|B|^2|C|}{\E(B, A-A)}>p^2 >|A|^{104/23}.
\]
Then using
\[
\E(B, A-A)\geq \frac{|B|^2|A-A|^2}{|B+A-A|}\geq  \frac{|B|^2|A-A|^2}{|A+A-A|}\gg\frac{|A|^4}{|A-A|},
\]
we see the required bound holds in this case.

To check the required $p$-constraint for the applications of  Lemma~\ref{lem:MirE4}, note that Lemma~\ref{lem:PRI1} implies
\[
|A||A+A-A||A+A-A-A|\leq \frac{|A-A|^{7}}{|A|^4}.
\]
Hence, if the condition of Lemma~\ref{lem:MirE4} (as applied to $\E_4(B, D))$ fails, we have $|A-A|\gg |A|^{28/23}$, which gives the required result. A similar analysis is necessary for the application of Lemma~\ref{lem:MirE4} to $\E_4(B)$; the ensuing required $p$-constraint is more forgiving and is already satisfied.
\section{Proof of Theorems~\ref{thm:BCS1}, \ref{thm:BCS2} and \ref{thm:LWP}}
We state some auxiliary exponential sum estimates, derived from basic applications of H\"older's inequality. See also \cite[Equation 3.7]{KonShp}.
\begin{lemma}
Let $X, Y\subseteq \F_q$. We have
\begin{equation}\label{eqn:BCSHLS}
    \left|\sum_{x\in X}\sum_{y\in Y}\psi(xy) \right|^4 \leq q |X|^3\E(Y).
\end{equation}
and
\begin{equation}\label{eqn:BCSHSS}
    \left|\sum_{x\in X}\sum_{y\in Y}\psi(xy) \right|^8 \leq q|X|^4|Y|^4\E(X)\E(Y).
\end{equation}
\end{lemma}

In addition, we will use the following result of Konyagin and Shkredov \cite[Lemma~4]{KonShk15}.
\begin{lemma}\label{lem:KonShkWi}
Let $X_1\subseteq X\subseteq \F_q$. Then
\[
\frac{1}{q}\sum_{y\in\F_q}\left|\sum_{x\in X}\psi(xy)\right|\geq \frac{|X_1|^2}{|X|^{1/2}\E(X_1)^{1/2}}.
\]
\end{lemma}

Let us record corollaries of Theorems~\ref{thm:bwdss} and \ref{thm:LEDRNSWI}, obtained through the same scheme as \cite[Lemma~4]{SwaWin}.
\begin{lemma}
\label{lem:SwaWinEB}
Let $f\in \F_p\left[x\right]$ denote a quadratic polynomial and let $T\subseteq \F_p$, with $|T|\leq p^{5/8}$ and the property that $f(T)\subseteq T$. Then there exists a subset $U\subseteq T$, with $|U| \gg |T|$ such that $\E(U)\lesssim |T|^{3-1/5}$.
\end{lemma}

\begin{proof}
Under the assumption $|T|\leq p^{5/8}$, Theorem~\ref{thm:bwdss} implies that $T = B\sqcup C$ such that 
\[
\max\{\E(B), \E(f(C))\} \lesssim |T|^{3-1/5}.
\]
Now, either $|B|\approx |T|$ or $|C|\approx |T|$. If the former is true, we may take $U = B$. If the latter is true, we take $U=f(C)$. Then $f(C)\subseteq f(T)\subseteq T$ by our assumption and clearly $|f(C)|\approx |T|$ as required.
\end{proof}

We also require the following analogue of the above, based on Theorem~\ref{thm:LEDRNSWI}. The proof is essentially the same as that of \cite[Lemma~4]{SwaWin}.
\begin{lemma}
\label{lem:SwaWinEBL}
Let $f\in \F_q(x)$ denote a rational function of degree $d$, satisfying \eqref{eqn:RNSHWINDC} and let $T\subseteq \F_q$ with the property that $f(T)\subseteq T$. Then there exists a subset $U\subseteq T$, with $|U| \gg |T|/((d+1)(\log|T|)^2)$ such that $\E(U)\lesssim |T|M_T^{-1}$, where $M_T$ is given by \eqref{eqn:MALS}.
\end{lemma}
\begin{proof}
By Theorem~\ref{thm:LEDRNSWI} there exist $C\subseteq B\subseteq T$, with $|T|\ll |B| \ll (\log|A|)^{2} |C|$ such that either of $\E(B)$ or $\E(f(C))$ is bounded by $O(|T|^3/M_T)$. If the former is true, we may take $U= B$. Otherwise, we take $U = f(C)\subset T$, noting that $|f(C)|\gg |C|/(d+1)\gg |T|/((d+1)(\log|T|)^2)$.
\end{proof}

\begin{proof}[Proofs of Theorems~\ref{thm:BCS1} and \ref{thm:BCS2}]
To prove Theorem~\ref{thm:BCS1} we use Lemma~\ref{lem:SwaWinEBL} to identify a subset $U\subseteq T$, with $|U|\gg |T|/((d+1)(\log|T|)^2)$. Then applying the bound on the additive energy of $U$, given by Lemma~\ref{lem:SwaWinEBL}, to \eqref{eqn:BCSHLS}, we obtain the required result.

Theorem~\ref{thm:BCS2} is proved similarly. To obtain \eqref{eqn:BCSDSB}, we use Lemma~\ref{lem:SwaWinEB} in place of Lemma~\ref{lem:SwaWinEBL}. To prove \eqref{eqn:BCSDSB2}, firstly through the same process, we identify a subset $V\subseteq S$, with $|V|\gg |S|$. Then we simply apply the bounds on the additive energy of $U$ and $V$ to \eqref{eqn:BCSHSS}.
\end{proof}

\begin{proof}[Proof of Theorem~\ref{thm:LWP}]
To prove \eqref{eqn:WienerfI}, we apply Proposition~\ref{prop:bwdss2} to the set $A$ and polynomial $g(x,y) = f(x) + f(y)$. Since $f$ is quadratic, the polynomial $g$ is non-degenerate. Note that $\E(f(A))\ll \E_g(A)$.
This ensures existence of sets $B, C\subseteq A$, with $|C|\gtrsim |B|\gg |A|$ such that
\[\E(B, A)^{3} \E(f(C))^2 \lesssim |A|^{14}.\]
Further note that $\E(B, A)\geq |A|^2|B|^2|B+A|^{-1}\gg |A|^4|A+A|^{-1}\gg |A|^{3}$. 
 
We use Lemma~\ref{lem:KonShkWi} with $X=f(A)$ and $X_1 = f(C)$ to get the required result. To check the required $p$-constraint of Proposition~\ref{prop:bwdss2},  note that $|A|^6 (\E(B,A))^{-1}\ll |A|^3$ and so our application is justified if $|A|^3\ll p^2$.

To prove \eqref{eqn:WienerOrbits}, we note that since $f(T)\subseteq T$, we may apply Lemma~\ref{lem:SwaWinEB}, which ensures existence of $U\subseteq T$ with $|U|\approx |T|$ and such that $\E(U)\lesssim |T|^{3-1/5}$. We apply Lemma~\ref{lem:KonShkWi}, with $X=T$ and $X_1 = U$, obtaining the required result.
\end{proof}

\section*{Acknowledgements}
We are especially grateful to Yiting Wang for pointing out an error in an earlier preprint.  
We thank Ilya Shkredov, Igor Shparlinski and Arne Winterhof for their helpful comments and suggestions.
The second author was supported by the Austrian Science Fund FWF grants P 30405 and P 34180.

\end{document}